\documentclass[a4paper,11pt,twoside]{article}

\usepackage{amssymb}
\usepackage[T1]{fontenc}
\usepackage[utf8]{inputenc}
\usepackage{babel}
\usepackage{mathrsfs}
\usepackage{amsthm}	
\usepackage{amsfonts,amsmath,lmodern,fancyhdr,lastpage,graphicx,nccfoots,caption}
\usepackage{afterpage}

\usepackage{indentfirst}
\usepackage{xcolor}
\usepackage{verbatim}

\usepackage{amssymb	
           ,bbm
            ,units 
           }
\usepackage{enumerate}
\usepackage[nobysame
           ,alphabetic
           ,initials
           ]{amsrefs}

\usepackage{pifont}
\usepackage{
            ulem		
           ,soul		
} 

\newtheorem{theorem}{Theorem}[section]
\newtheorem{lemma}[theorem]{Lemma}
\newtheorem{proposition}[theorem]{Proposition}

\newtheorem{definition}[theorem]{Definition}

\newtheorem{remark}[theorem]{Remark}
\newtheorem{example}[theorem]{Example}
\newtheorem{assumption}[theorem]{Assumption}

\newcommand{\bD}{\mathbb{D}}
\newcommand{\bE}{\mathbb{E}}
\newcommand{\bF}{\mathbb{F}}

\newcommand{\bN}{\mathbb{N}}
\newcommand{\bP}{\mathbb{P}}

\newcommand{\bR}{\mathbb{R}}
\newcommand{\bS}{\mathbb{S}}
\newcommand{\bT}{\mathbb{T}}

\newcommand{\cB}{\mathcal{B}}
\newcommand{\cC}{\mathcal{C}}

\newcommand{\cF}{\mathcal{F}}

\newcommand{\cH}{\mathcal{H}}

\newcommand{\cP}{\mathcal{P}}

\newcommand{\cS}{\mathcal{S}}

\DeclareMathOperator*{\esssup}{ess\,sup}

\newcommand{\trace}{\textrm{Trace}}
\newcommand{\Supp}{\textrm{Supp}}

\definecolor{darkgreen}{rgb}{0,0.35,0}

\newcommand{\1}{\mathbbm{1}}

\newcommand{\dd}{d}

\usepackage{hyperref}

\vfuzz \hfuzz
\newcounter{a}
\ifodd\thea\else\stepcounter{a}\fi

\fancypagestyle{plain}{%
\fancyhf{}

}
\defineshorthand{"-}{\nobreak\hskip0pt\discretionary{-}{-}{-}\nobreak\hskip0pt}

\begin{document}
\thispagestyle{empty}

\begin{center}
\Large
\textsc{Malliavin differentiability of McKean-Vlasov SDEs with locally Lipschitz coefficients}
\end{center}


\begin{center}
\textit{Gon\c calo dos Reis}\par
\begin{tabular}{c}
University of Edinburgh, School of Mathematics, Edinburgh,  UK\\
and\\
Center for Mathematics and Applications (NOVA Math),\\
NOVA School of Science and Technology (NOVA FCT), PT
\\
	e-mail: \texttt{G.dosReis@ed.ac.uk
	}
\end{tabular}
\end{center}


\begin{center}
\textit{Zac Wilde} \par
\begin{tabular}{c}
University of Edinburgh, School of Mathematics, Edinburgh, UK\\
	e-mail: \texttt{Z.J.Wilde@sms.ed.ac.uk
	}
\end{tabular}
\end{center}




\noindent
\textbf{Abstract:}  
In this short note, we establish Malliavin differentiability of McKean-Vlasov Stochastic Differential Equations (MV-SDEs) with drifts satisfying both a locally Lipschitz and a one-sided Lipschitz assumption, and where the diffusion coefficient is assumed to be uniformly Lipschitz in its variables.  

As a secondary contribution, we investigate how Malliavin differentiability transfers across the interacting particle system associated with the McKean-Vlasov equation to its limiting equation. This final result requires both spatial and measure differentiability of the coefficients and doubles as a standalone result of independent interest since the study of Malliavin derivatives of weakly interacting particle systems seems novel to the literature. The presentation is didactic and finishes with a discussion on mollification techniques for the Lions derivative. 

\medskip


\tableofcontents

\section{Introduction}
The main object of our study are McKean-Vlasov Stochastic Differential Equations (MV-SDE), also known as mean-field equations or distribution-dependent SDEs. They differ from standard SDEs by means of the presence of the law of the solution process in the coefficients. Namely
\begin{align*}
\dd Z_{t} = b(t,Z_{t}, \mu_t)\dd t + \sigma(t,Z_{t}, \mu_t)\dd W_{t},  \quad X_{0} =\xi,
\end{align*}
for some measurable coefficients, where $\mu_{t}=\textrm{Law}(Z_t)$ denotes the law of process $Z$ at time $t$ (with $W$ a Brownian motion; $\xi$ a random initial condition). Similar to standard SDEs, MV-SDEs are shown to be well-posed under a variety of frameworks, for instance, under  locally Lipschitz and super-linear growth conditions alongside random coefficients, see e.g. \cite{adamssalkeld2022LDPReflected} or \cite{DosReis2018-LDPforMVSDEs}. In this setting there are also many studies on their numerical approximation e.g. \cites{MR4367675,MR4413221,imanumdrad022}, ergodicity \cite{chen2023wellposedness} and large deviations \cites{DosReis2018-LDPforMVSDEs,adamssalkeld2022LDPReflected}.

Many mean-field models exhibit drift dynamics that include superlinear growth and non-global Lipschitz growth, for example, mean-field models for neuronal activity (e.g.~stochastic mean-field FitzHugh-Nagumo models or the network of Hodgkin-Huxley neurons) \cite{BaladronFasoliFaugerasEtAl2012}, \cite{Bolley2011}, \cite{BossyEtAl2015} appearing in biology or the physics of modelling batteries \cite{dreyer2011phase},\cite{DreyerFrizGajewskiEtAl2016}. Quoting \cite{gomes2019mean}, systems of weakly-interacting particles and their limiting processes, so-called McKean-Vlasov or mean-field equations appear in a wide variety of applications, ranging from plasma physics and galactic dynamics  to mathematical biology, the social sciences, active media, dynamical density functional theory (DDFT) and machine learning. They can also be used in models for co-operative behavior, opinion formation, risk management, as well as in algorithms for global optimization.

\smallskip 
\emph{Our 1st contribution: Malliavin differentiability of MV-SDEs under locally Lipschitz conditions.} We extend Malliavin variational results to McKean-Vlasov SDEs with locally Lipschitz drifts satisfying a so-called one-sided Lipschitz condition. The result is new to the best of our knowledge. Malliavin differentiability of MV-SDEs has been addressed in \cite{Crisan2018}*{Proposition 3.1} and \cite{RenWang2018}, and in both cases, their assumptions revolve around the differentiable Lipschitz case. Our proof methodology is inspired by that of \cite{Crisan2018} --  both there and here, the result is established by appealing to the celebrated \cite{nualart2006malliavin}*{Lemma 1.2.3}.

\emph{Our 2nd contribution: transfer of Malliavin differentiability across the particle system limit}.
Another large aspect of McKean-Vlasov SDE theory, is the study of the large weakly-interacting particle systems and their particle limit that recovers the MV-SDE in the limit. This latter limit result is called Propagation of Chaos \cite{Sznitman1991} (also \cite{Bolley2011},\cite{gomes2019mean}, \cite{chassagneux2014probabilistic}, \cite{adamssalkeld2022LDPReflected}). In the second part of this note, we study the Malliavin differentiability of the interacting particle system and how the Malliavin regularity transfers across the particle limit to the limiting equation. To the best of our knowledge, this particular proof methodology is new to the literature. 

From a methodological viewpoint, our point of attack is the \emph{projections over empirical measures} approach \cites{chassagneux2014probabilistic,book:CarmonaDelarue2018a,platonovdosreis2023ItowentzelLions}. This approach allows us to use the best available Malliavin differentiability results for standard (multidimensional) SDEs \cites{nualart2006malliavin,imkeller2018differentiability}, and then carry them to the MV-SDE setting via the particle limit using Propagation of Chaos and \cite{nualart2006malliavin}*{Lemma 1.2.3}. Our variational results are limited only by the SDE results we cite. If better results are found, one only needs to replace the reference in the appropriate place. 
Lastly, in relation to our 1st contribution, this 2nd contribution is established under a full global Lipschitz and differentiability (space and measure) assumption on the coefficients. 
\medskip

\textbf{Organization of the paper.} In section 2, we set notation and review a few concepts necessary for the main constructions. In section 3, we prove Malliavin differentiability of MV-SDEs under superlinear drift growth assumptions and in section 4 we prove Malliavin differentiability of MV-SDEs under the weaker global Lipschitz assumptions via the convergence of interacting particle systems, providing a lengthy remark about mollification in Wasserstein spaces.

\section{Notation and preliminary results}
\label{sec:two}

\subsection{Notation and Spaces}

For collections of vectors, let the upper indices denote the distinct vectors, whereas the lower index is a vector component, i.e. $x^l_j$ denote the $j$-th component of $l$-th vector. Let $\mathbf{x} = (x^1,\cdots, x^N)$ denote a vector in $\bR^{dN}$ where $x^i := (x^i_1,\cdots,x^i_d)$ for $i=1,\cdots ,N$. 
\color{black} 
For matrices $M$ and $N$ of agreeing dimensions, define the inner product $M:N = \trace(M^TN)$  and the norm induced by this inner product (the Hilbert-Schmidt norm) as $|M|=\sqrt{\trace(M^TM)}$. 

 For $g:\bR^m\to\bR^n$, writing $g(x) = (g^1(x),\cdots,g^n(x))$ and define $\nabla_x g$ as the Jacobian matrix 
$(g_{ij})_{\substack{1\leq i \leq m \\ 1\leq j\leq n}}$ 
where $g_{ij} =\partial_{x^i}g^j$. 
\color{black}

Take $T\in[0,\infty)$ and let $(\Omega, \bF, \cF, \bP)$ be a filtered probability space carrying a $m$-dimensional Brownian Motion on the interval $[0,T]$ and $\bF=(\cF_t)_{t\geq 0}$. The filtration is the one generated by the Brownian motion and augmented by the $\bP$-null sets, and with an additionally sufficiently rich sub $\sigma$-algebra $\cF_0$ independent of $W$. We denote by $\bE[\cdot]=\bE^\bP[\cdot]$ the usual expectation operator with respect to $\bP$.  

The space of probability measures on $\bR^d$ with finite second moment, $\cP_2(\bR^d)$ is Polish under the $2$-Wasserstein distance
\begin{align*}
W_2(\mu,\nu) = \inf_{\pi\in\Pi(\mu,\nu)} \Big(\int_{\bR^d\times \bR^d} |x-y|^2\pi(dx,dy)\Big)^\frac12, \quad \mu,\nu\in \cP_2(\bR^d),
\end{align*}   
where $\Pi(\mu,\nu)$ is the set of couplings for $\mu$ and $\nu$ such that $\pi\in\Pi(\mu,\nu)$ is a probability measure on $\bR^d\times \bR^d$ such that $\pi(\cdot\times \bR^d)=\mu$ and $\pi(\bR^d \times \cdot)=\nu$. 
Let $\Supp(\mu)$ denote the support of $\mu \in \cP(\bR^d)$.

Let $p\in[2,\infty)$. We introduce the following spaces.
\begin{itemize}
\item Let $\mathcal{X}$ be a metric space. We denote by $C(\mathcal{X})$ as the space of continuous functions $f:\mathcal{X} \to \bR$ endowed with the uniform norm and $C_b(\mathcal{X})$ its subspace of bounded functions endowed with the sup norm $\|f\|_\infty = \sup_{x\in \mathcal{X}}|f(x)|<	\infty$; 
For $k\in \bN$ denote $\cC^k(\bR^d)$ the space of $k$-times continuously differentiable functions from $\bR^d$ to $\bR^d$, equipped with a collection of seminorms $\{ \|g\|_{\cC^p(K)} := \sup_{x \in K} ( |g(x)| + \sum_{j=1}^k |\partial_x^j g(x)|), g\in \cC^k(\bR^d) \}$, indexed by the compact subsets $K \subset \bR^d$.

\item 
$L^{p}(\Omega):=L^{p}(\Omega, \cF_t, \bP; \bR^d)$ \index{$L^p$}, $t\in [0,T]$ is the space of $\bR^d$-valued $\cF_t$-measurable random variables $X:\Omega\to\bR^d$ with norm  $\|X\|_{L^p(\Omega)} = \bE[\, |X|^p]^{1/p} < \infty$.

\item  
$\cS^p([0,T]):=\cS^{p}([0,T], \bP;\bR^d)$ is the space of $\bR^d$-valued measurable $\bF$-adapted processes $(Y_t)_{t\in[0,T]}$ satisfying $\|Y \|_{\cS^p([0,T])} = \bE[\sup_{t\in[0,T]}|Y(t)|^p]^{1/p}
<\infty$. 
\end{itemize}

\subsection{Malliavin Calculus}
Let $\cH$ be a Hilbert space and $W:\cH \to L^2(\Omega)$ a Gaussian random variable. The space $W(\cH)$ endowed with an inner product $\langle W(h_1), W(h_2)\rangle = \bE[W(h_1) W(h_2)]$ is a Gaussian Hilbert space. \color{black}Let $C_p^\infty(\bR^n; \bR)$ be the space of infinitely differentiable functions $f:\bR^n \to \bR$ which have partial derivatives of all orders, each with polynomial growth. \color{black} Let $\bS$ be the collection of random variables $F:\Omega \to \bR$ such that for $n\in\bN$, $f\in C_p^\infty(\bR^n; \bR)$ and $h_i \in \cH$ can be written as $F = f(W(h_1), \ldots , W(h_n))$. Then we define the derivative of $F$ to be the $\cH$-valued random variable
\begin{align*}
    DF = \sum_{i=1}^n \partial_{x_i} f(W(h_1), \ldots, W(h_n)) h_i.
\end{align*} 

The Malliavin derivative from $L^p(\Omega)$ into $L^p(\Omega, \cH)$ is closable and the domain of the operator is defined to be $\bD^{1, p}$, defined to be the closure of the of the set $\bS$ with respect to the norm
\begin{align*}
    \|F\|_{1, p} = \Big[ \bE[ |F|^p] + \bE[ \|DF\|_{\cH}^p] \Big]^{\tfrac{1}{p}}. 
\end{align*}
\color{black}
We also define the directional Malliavin derivative $D^hF = \langle DF, h\rangle_\cH$ for any choice of $h\in\cH$.  For more details, see \cite{nualart2006malliavin}. 
\color{black}
\section{Malliavin differentiability under local Lipschitz assumptions}
\label{sec:MallDiffgeneral}

\subsection{McKean-Vlasov Equations with locally Lipschitz coefficients}

In this manuscript, we work with so-called McKean-Vlasov SDEs described by the following dynamics for  $0\leq t \leq T<\infty$,
\begin{align}
\label{Eq:General MVSDE}
\dd Z_{t} = b(t,Z_{t}, \mu_t)\dd t + \sum_{l=1}^m\sigma^l(t,Z_{t}, \mu_t)\dd W_{t}^l,  \quad Z_{0} =\xi,
\end{align}
where $\mu_{t}$ denotes the law of the process $Z$ at time $t$, i.e.~$\mu_t=\bP\circ Z_t^{-1}$ and $W^l$, $l=1,\ldots,m$ are $1$-dimensional independent Brownian motions. We write $W=(W^1,\ldots,W^m)$ as the corresponding $m-$dimensional Brownian motion. In this chapter, we work under the locally Lipschitz case as the below assumption describes.

\subsubsection{Assumptions}
\begin{assumption}
\label{ass:MKSDE-MainExistTheo}
Let $b:[0,T] \times \bR^d \times\cP_2(\bR^d) \to \bR^d$ and for $l = 1,\ldots,m$, $\sigma^l:[0,T] \times \bR^d \times \cP_2(\bR^d) \to \bR^{d}$. Then there $\exists L>0$ such that: 
\begin{enumerate}
\item \color{black}For some $p\geq 2$, $\xi$ is $\cF_0$-measurable and $\xi \in L^p(\Omega, \cF_0, \bP; \bR^d)$ .
\color{black}
\item $\sigma^l$ is continuous in time and Lipschitz in space-measure $\forall t \in[0,T]$, $\forall x, x'\in \bR^d$ and $\forall \mu, \mu' \in \mathcal{P}_2(\bR^d)$ we have
\begin{align*}
    |\sigma^l(t,  x, \mu) - \sigma^l(t , x', \mu')| \leq L\Big( |x-x'| +W_2(\mu, \mu')\Big).
\end{align*}
\item $b$ is continuous in time and satisfies the \emph{one-sided Lipschitz condition} in space and is Lipschitz in measure: $\forall t \in[0,T]$, $\forall x,x'\in \bR^d$ and $\forall \mu, \mu' \in \cP_2(\bR^d)$ we have that 
\begin{align*}
\big\langle x-x', b(t , x, \mu)-b(t, \omega, x', \mu)\big\rangle_{\bR^d} & \leq L|x-x'|^2,
\\
|b(t, x, \mu) - b(t , x, \mu')| & \leq L W_2 (\mu, \mu').
\end{align*}
\item $b$ is Locally Lipschitz: $\forall t \in[0,T]$, $\forall \mu \in \cP_2(\bR^d)$, $\forall x, x'\in\bR^d$ such that $|x|, |x'|<N$ we have that $\exists L_N>0$ such that 
\begin{align*}
    |b(t, x, \mu) - b(t , x', \mu) | \leq L_N |x-x'|.
\end{align*}
\end{enumerate}
\end{assumption}
Throughout, denote by $\sigma$, the $d\times m$ matrix with columns $(\sigma^1,\cdots, \sigma^{m})$.\\

Observe that time continuity is a sufficient condition for integrability of $b$ and $\sigma^l$ since we are working on a compact time interval. The above assumption ensures existence, uniqueness and related stability while the following will be used to ensure the differentiability results.

\begin{assumption} 
\label{Assumption:MalCalcMV}
Let Assumption \ref{ass:MKSDE-MainExistTheo} hold.  
For any $t\geq 0$, $\mu\in\cP(\bR^d)$ the maps $x\mapsto b(t,x,\mu)$ and $x\mapsto \sigma^l(t,x,\mu)$ are $C^1(\bR^d)$. The derivative maps are jointly continuous in their variables. 
\end{assumption}

\begin{remark}\label{remark:Extensions}
One recognises that the above assumptions can be weakened in a few ways. For instance, the Lipschitz constant can be a non-negative function of time $L_t$, under an  $L^1$-integrability condition: $\int_0^T L_s\dd s< \infty$. 
Further, the time continuity $t\mapsto b(t,0,\delta_0)$ can be exchanged for an integrability condition: $\int_0^T |b(s,0,\delta_0)|\dd s <\infty$, while the time-continuity of $t\mapsto \sigma(t,0,\delta_0)$ can be exchanged for a square-integrability condition: $\int_0^T|\sigma(s,0,\delta_0)|^2\dd s <\infty$. We leave these points open for the interested reader. 
\end{remark}

\subsubsection{Well-posedness and moment estimates}

The first result establishes well-posedness, moment estimates and continuity in time for the solution of \eqref{Eq:General MVSDE}. 
\begin{theorem}
    \label{theo:Wellposendess} 
Let Assumption \ref{ass:MKSDE-MainExistTheo} hold with some $p\geq 2$. Then, MV-SDE \eqref{Eq:General MVSDE} is well-posed and has a unique solution $Z\in \cS^p([0,T])$. Moreover, it satisfies 
\begin{align*}
    \mathbb{E}\left[\sup _{t \in[0, T]}|Z_t|^p\right]  
    \leq Ce^{CT}\bigg(\bE\big [\, |\xi|^p \big]
    +\bE\bigg[\bigg(
    &
    \int_0^T |b(s,0,\delta_0)|\dd s\bigg)^p\bigg]
    \\
    &
    +\bE\bigg[\bigg(\int_0^T |\sigma(s,0,\delta_0)|^2\dd s\bigg)^\frac{p}{2}\bigg]
    \bigg),
\end{align*}
for some positive constant $C$. Lastly, $Z$ has $\bP$-almost surely continuous paths and its law $[0,T]\ni t\mapsto \mu_t$ is continuous under the $W_2$-distance.
\end{theorem}
This result also yields estimates for standard SDEs (which have no measure dependency).
\begin{proof}
Well-posedness and the moment estimate follow from  \cite{adamssalkeld2022LDPReflected}*{Theorem 3.2} as their assumption (with $\mathcal{D}=\bR^d$, $x_0$ and deterministic continuous maps $b,\sigma$) subsumes our Assumption \ref{ass:MKSDE-MainExistTheo}.
The continuity of the sample paths of $Z$ and its law in $W_2$ is trivial.
\end{proof}

\subsection{Malliavin differentiability with locally Lipschitz coefficients}
\label{section:Mall Diff superlinear growth}

We state the first main result of this work, the Malliavin differentiability of the solution of \eqref{Eq:General MVSDE}. 

\begin{theorem}
\label{proposition:MallDiff-all-around}
Let $p\geq 2$. Let Assumption \ref{Assumption:MalCalcMV} hold. Denote by $Z$ the unique solution \eqref{Eq:General MVSDE} in $\cS^p([0,T])$. 
Then $Z$ is Malliavin differentiable, i.e. $Z\in \bD^{1,2}(\cS^2) \cap \bD^{1,p}(\cS^p)$, and the Malliavin derivative against $W$ satisfies for $0\leq s\leq t \leq T$,
\begin{align}
\label{eq:MVSDEMallDeriv}
D_s Z_t
= \sigma(s,Z_s, \mu_s)
&
+ \int_s^t (\nabla_x b)
(r, Z_r, \mu_r)
D_s Z_r dr \notag
\\
&
+ \sum_{l=1}^m\int_s^t (\nabla_x \sigma^l)(r, Z_r, \mu_r) D_s Z_r dW_r^l.
\end{align}
If $s >t$ then $D_s Z_t=0$ $\bP$-a.s. 

Moreover, we have 
\begin{align*}
\sup_{0\leq s\leq T} \|D_s Z \|_{\cS^p([0,T])}^p
\leq 
\sup_{0\leq s\leq T} \bE\Big[\sup_{0\leq t \leq T}|D_s Z_t|^p\Big]
&\leq C(1+\|Z\|_{\cS^p([0,T])}^p)
\\
&\leq C\big( 1+\|\xi\|_{L^p(\Omega)}^p \big)
 <\infty,
\end{align*}
\color{black}
and, reflecting that $DZ_t$ is a Hilbert space valued random variable we have   
\begin{align}
\label{eq:aux:HilbertSpaceDominatoin}
\bE\Big[\sup_{0\leq t \leq T} \big(\int_0^T |D_s Z_t|^2ds \Big)^{\frac p2} \Big]
&\leq C\big(1+\|\xi\|_{L^p(\Omega)}^p \big)
 <\infty.
\end{align}
\color{black}
\end{theorem}
This proof is inspired by that appearing in \cite{Crisan2018} but with critical differences to allow for the superlinear growth of the drift and the general Malliavin differentiability of \cite{imkeller2018differentiability}. 

We comment that using the proof methodology we present below, the result above can be extended in several ways -- we leave these as open questions. The first is to allow for random drift and diffusion coefficients: \cite{adamssalkeld2022LDPReflected}*{Theorem 3.2} provides well-posedness and moment estimates and concluding Malliavin differentiability via \cite{imkeller2018differentiability}*{Theorem 3.2 or Theorem 3.7}. 

Another open setting, with continuous deterministic coefficients, is to establish our Theorem \ref{proposition:MallDiff-all-around} with a drift map $b$ having super-linear growth in the measure component: for instance, by allowing convolution type measure dependencies. Lastly, the differentiability requirements for $b$ and $\sigma$ can be weakened via mollification; see \cite{imkeller2018differentiability}*{Remark 3.4}.

\begin{example}[Linear interaction kernels]
Take MV-SDE \eqref{Eq:General MVSDE} with solution $(Z_t,\mu_t)_{t\geq 0}$ under the assumptions of Theorem \ref{proposition:MallDiff-all-around} and let the drift function $b$ take a specific convolutional form. 
Concretely, let $b(t,x,\mu) = (\Tilde{b}* \mu )(x)= \int_{\bR^d}\Tilde{b}(x-y)\mu(\dd y)$ for some $\Tilde{b} \in C^1(\bR^d)$  and let $\sigma = \sigma_0 I_d$ for $\sigma_0\neq 0$ a constant. 

Then, we have for any $0\leq s\leq t \leq T$
\begin{align}
\label{mall example}
    D_sZ_t = \sigma_0 \exp\Big(\int_s^t(\nabla_x\tilde{b}*\mu_r)(Z_r)\dd r\Big).
\end{align}
We can in fact reduce the $C^1$ in space differentiability assumption to a Lipschitz one, exploiting \cite{nualart2006malliavin}*{Proposition 1.2.4}. That is, we can take a sequence of mollifiers $\Tilde{b}_n:=\Tilde{b}*\rho_n$ for a smoothing kernel $(\rho_n)_n$ such that   $C^\infty(\bR^d) \ni \rho_n (x)\rightarrow x$ uniformly, and the expression \eqref{mall example} still holds.
\medskip 

We can generalise the above form of the drift to linear interaction kernels of the form $b(t,x,\mu) = \int_{\bR^d}\widehat{b}(x,y)\mu(\dd y)$, where 
$\widehat{b}:\bR^d\times\bR^d\to \bR^d $ and $ x\mapsto \widehat{b}(x,\cdot)$ is Lipschitz uniformly to obtain for any $0\leq s\leq t \leq T$ 
\begin{align*}
    D_sZ_t = \sigma_0 \exp\Big(\int_s^t \Big\{  \int_{\bR^d} (\nabla_x\widehat {b})(Z_r,y) \mu_r (dy) \Big\}\dd r\Big).
\end{align*}
\end{example}

\begin{proof}[Proof of Theorem \ref{proposition:MallDiff-all-around}]
The Malliavin differentiability of \eqref{Eq:General MVSDE} is shown by appealing to \cite{nualart2006malliavin}*{Lemma 1.2.3}. 
One builds a convenient sequence of Picard iterations which converge to the McKean-Vlasov Equation and use \cite{nualart2006malliavin}*{Lemma 1.2.3} to ensure that the limit is also Malliavin differentiable in $\bD^{1,2}(\cS^2)$. Lastly, by showing that the Malliavin derivative of $Z$ is $\cS^p$-integrable, then  \cite{nualart2006malliavin}*{Proposition 1.5.5} yields that $Z\in \bD^{1,p}(\cS^p)$.

\emph{Step 1.0. The Picard sequence.} 
We start by defining a Picard sequence approximation for \eqref{Eq:General MVSDE}, namely set $Z^0_\cdot=\xi$ and $\mu^0_t= \bP\circ \xi^{-1}=\textrm{Law}(\xi)$ for any $t\geq 0$; we have $t\mapsto \mu^0_t$ is a $W_2$-continuous map. For any $n\geq 1$ define 
\begin{align}
\label{eq:StylizedDecoupledMVSDE}
\dd Z^{n+1}_{t} = b(t,Z^{n+1}_{t}, \mu^n_t)\dd t + \sum_{l=1}^m\sigma^l(t,Z^{n+1}_{t}, \mu^{n}_t)\dd W_{t}^l,  \quad Z^{n+1}_{0} =\xi.
\end{align}
 \eqref{eq:StylizedDecoupledMVSDE} is a standard SDE with added time dependence induced by $t\mapsto \mu^n_t$ with drift $b$ satisfying a one-sided Lipschitz condition (in space) and $\sigma$ uniformly Lipschitz (in space). \\

\emph{Step 1.1. Existence and uniqueness of $Z^n$.} 
Take $\mu^{n-1}$ such that $t\mapsto b(t,x,\mu^{n-1}_t)$ and $t\mapsto \sigma^l(t,x,\mu^{n-1}_t)$ are continuous, then (a slight variation of) Theorem \ref{theo:Wellposendess}, given Assumption \ref{ass:MKSDE-MainExistTheo}, yields the existence of a unique solution $Z^{n+1}\in \cS^2([0,T])$. 
Moreover, an easy variation of \cite{DosReis2018-LDPforMVSDEs}*{Proposition 3.4} yields that for $n\geq 0$, $t\mapsto \mu^{n+1}_t$ is continuous (in 2-Wasserstein distance) if $t\mapsto \mu^{n}_t$ is.  
We can conclude that  $\{Z^n\}_{n\geq 0}$ exists and is well defined. 

Using that $W_2(\delta_0 ,\mu^n_\cdot)^2\leq \bE[|Z^{n}_\cdot|^2]$ 
and Theorem \ref{theo:Wellposendess}, we have 
\begin{align*}
\| & Z^{n+1} \|_{\cS^2([0,T])}^2
\\
& 
\leq 
C\Big(1+\|\xi\|_{L^2(\Omega)}^2 
            + \int_0^T \bE[|Z^{n}_r|^2]dr\Big)
\\
&
\leq 
C\Big(1+\|\xi\|_{L^2(\Omega)}^2 
       + \int_0^T \|Z^{n}\|_{\cS^2([0,r])}^2 dr\Big)
\\
&
\leq 
C\Big(1+\|\xi\|_{L^2(\Omega)}^2 + \int_0^T \Big\{
C(1+\|\xi\|_{L^2(\Omega)}^2 + \int_0^r \|Z^{n-1}\|_{\cS^2([0,s])}^2ds\Big\} dr\Big)
\\
&
\leq \cdots 
\leq C\big(1+\|\xi\|_{L^2(\Omega)}^2\big)\Big( \sum_{j=0}^{n} \frac{(CT)^j}{j!} \Big) \|Z^{0}\|_{\cS^2([0,T])}^2 
\\
& \hspace{1cm}
\leq  C\big(1+\|\xi\|_{L^2(\Omega)}^2\big)e^{CT} \|Z^{0}\|_{\cS^2([0,T])}^2,
\end{align*}
where we iterated the initial estimate on $[0,T]$ over small subintervals $[0,r]$ leading to a known \emph{simplex} estimate. We conclude that 
\begin{align}
 \label{auxeq:unifBoundineqforXn-Malldiff}
\sup_{n\geq 0} \big\{ \|Z^n\|_{\cS^2([0,T])}+\sup_{0\leq t\leq T} W_2(\delta_0,\mu^n_t)\big\} <\infty.
\end{align}
\emph{Step 1.2. Convergence of $Z^n$.} 
Recall that \eqref{eq:StylizedDecoupledMVSDE} is a standard SDE, thus standard SDE stability estimation arguments apply. 
We sketch such argument only and invite the reader to inspect the proof of \cite{DosReis2018-LDPforMVSDEs}*{Proposition 3.3} for the full details. 
Take the SDE for the difference of $Z^{n+1}-Z^n$, i.e. 
\begin{align*}
Z^{n+1}_{t} -Z^{n}_{t} 
=  
\int_0^T \big[ b(s, & Z^{n+1}_{s}, \mu^n_s)
-
b(s,Z^{n}_{t}, \mu^{n-1}_s)\big]\dd s
\\
& +
\sum_{l=1}^m\int_0^T \big[
\sigma^l(s,Z^{n+1}_{s}, \mu^n_s)
-\sigma^l(s,Z^{n}_{t}, \mu^{n-1}_s)\big]\dd W_{s}^l.
\end{align*}
Applying It\^o's formula to $|Z^{n+1}_{t} -Z^{n}_{t}|^2$, using the growth assumptions on $b,\sigma$ and taking the supremum over time and expectations, we use the BDG inequality and then the Gr\"onwall inequality to obtain  
\begin{align*}
\|Z^{n+1}&-Z^n\|_{\cS^2([0,T])}^2
\\
&\leq C_T \big( \int_0^T W_2(\mu^n_r,\mu^{n-1}_r)^2dr\big)
\leq C_T \int_0^T \bE[|Z^{n}_r-Z^{n-1}_r|^2]] dr
\\
&\leq C_T \int_0^T \| Z^{n}-Z^{n-1}\|^2_{\cS^2([0,r])} dr
\leq \cdots\leq 
\frac{(C_T T)^n}{n!} \| Z^{1}-Z^{0}\|^2_{\cS^2([0,T])},
\end{align*}
where we used the same \emph{simplex} trick as in \textit{Step 1.1} above. We conclude that $Z^n$ converges to the solution of \eqref{Eq:General MVSDE} in $\cS^2([0,T])$ as $\| Z^{1}-Z^{0}\|^2_{\cS^2([0,T])}$ is bounded and independent of $n$.
 \medskip 

\emph{Step 2. Malliavin differentiability for $Z^n$.} Now, under our Assumption \ref{Assumption:MalCalcMV}, \cite{imkeller2018differentiability}*{Corollary 3.5} holds applied to yield the Malliavin differentiability of $Z^n$ for each fixed $n$; critically, due to the equation's coefficients being deterministic, that corollary does not require $p>2$ and it holds for any $p\geq 2$ (in particular, our case here for the time being $p=2$). 
We have that the Malliavin Derivative $D Z^{n+1}$ satisfies $D_s Z^{n+1}_{t}=0$ for $0\leq t<s\leq T$, while for $0\leq s\leq t\leq T$ it is given by the SDE dynamics 
\begin{align*}
D_s Z^{n+1}_t 
= 
\sigma(s, Z^{n+1}_s, \mu^{n}_s) 
& +\int_s^t (\nabla_x b)(r, Z^{n+1}_r, \mu^{n}_r) D_s Z^{n+1}_r dr 
\\
&+ \sum_{l=1}^m\int_s^t (\nabla_x \sigma^l)(r, Z^{n+1}_r, \mu^{n}_r) D_s Z^{n+1}_r dW_r^l.
\end{align*}
\smallskip 

\emph{Step 3. Uniform bound on $D Z^n$.} 
We remark that the SDE for $D Z^n$ is linear and satisfies \cite{imkeller2018differentiability}*{Assumption 2.4} with their $b(s,\omega),\sigma(s,\omega)$ set to zero, hence \cite{imkeller2018differentiability}*{Theorem 2.5} applies to yield 
\begin{align}
\nonumber 
\|D_\cdot Z^{n} \|_{\cS^2([0,T])}^2
= 
\bE\big[\, &\sup_{0\leq t\leq T}|D_\cdot Z^n_t|^2\, \big]
\leq 
C \bE\big[\, 
|\sigma(\cdot, Z^{n}_\cdot,\mu_\cdot^{n-1})|^2\, \big]
\\
\nonumber 
& \hspace{-1cm}
\leq 
C\Big(1+\|Z^n\|_{\cS^2([0,T])}^2 + W_2^2(\mu^{n-1}_\cdot,\delta_0) \Big)
\\
\nonumber 
& \hspace{-1cm}
\leq 
C\Big(1+\|Z^n\|_{\cS^2([0,T])}^2 +\|Z^{n-1}\|_{\cS^p([0,T])}^2\Big )
\\
\label{eq:auc:yetanotherrandomlabel}
& \hspace{-1cm}
\leq C(1+\|\xi\|^2_{L^2(\Omega)})
 <\infty,
\end{align}
where we used the linear growth of $\sigma$ and its time continuity property. The constant $C$ depends heavily on the constants appearing in Assumption \ref{ass:MKSDE-MainExistTheo} and the upper bound on $t \mapsto \sigma(t,0,\delta_0)$ stemming from its continuity over the compact  $[0,T]$.  
Using that \eqref{auxeq:unifBoundineqforXn-Malldiff} provides an estimate of $\|Z^n\|_{\cS^2([0,T])}$ uniform over $n$ and uniform over the Malliavin derivative parameter, we conclude taking supremum that    
\begin{align}
\label{eq:aux: unifrom estimates}
    \sup_{n\in\bN} \sup_{s\in[0,T]}\bE[\, \sup_{t\in[0,T]}  |D_s Z^n_t|^2 \,  ] 
    \leq C(1+\|\xi\|^2_{L^2(\Omega)})
    <\infty.
\end{align}

\color{black}
We also establish the inequality for $DZ^n_t$ that yields its interpretation as Hilbert space valued random variable. 
Simple manipulations and using \eqref{eq:aux: unifrom estimates} above easily yields estimate \eqref{eq:aux:HilbertSpaceDominatoin} when $p=2$ (when we pass to the limit in $n \to \infty$).  Concretely, we have   
\begin{align}
\notag 
\bE\Big[ \sup_{t\in [0,T]} \int_0^T |D_s Z_t^n|^2 ds \Big]
&
\leq
C \bE\Big[  \int_0^T  \big\{\sup_{t\in [0,T]} |D_s Z_t^n|^2  \big\}  ds  \Big]
\\ \notag
& 
\leq C 
\int_0^T \bE\Big[ \sup_{t\in [0,T]} |D_s Z_t^n|^2  \Big] ds
\\ \notag
& 
\leq C 
\sup_{s\in [0,T]} \bE\Big[ \sup_{t\in [0,T]} |D_s Z_t^n|^2 \Big]
\\ \label{eq:aux:Hilbert}
& 
   \leq C(1+\|\xi\|^p_{L^2(\Omega)})
    <\infty, 
\end{align}
where we firstly used Jensen's inequality and moved the supremum inside the integral, then used Fubini and dominated the integral by the supremum over the integration variable in a way that  \eqref{eq:auc:yetanotherrandomlabel} can be used. 
Taking supremum over $n$ on both sides and using \eqref{eq:aux: unifrom estimates} yields \eqref{eq:aux:HilbertSpaceDominatoin}.
\color{black}

By applying \cite{nualart2006malliavin}*{Lemma 1.2.3}, we conclude that the limit of $Z$ is Malliavin differentiable and the limit of $DZ^n$ gives its Malliavin derivative (identified by Equation \eqref{eq:MVSDEMallDeriv}).  
The moment estimates for $DZ^n$, holding uniformly over $n$, yield the moment estimate for $DZ$.
\medskip

\color{black}
\emph{Step 4. Higher-order moments on $D Z$ and conclusion for $\bD^{1,p}(\cS^p)$.} 

Recall that $p\geq 2$. From Theorem \ref{theo:Wellposendess} we have that $Z\in \cS^p([0,T])$ and the estimates 
\begin{align*}
\|Z\|_{\cS^p([0,T])}^p 
\leq C\big(1+\|\xi\|_{L^p(\Omega)}^p \big)
\quad \textrm{and}\quad 
\sup_{0\leq t\leq T } W_2^2(\mu_t,\delta_0) \leq C\big(1+\|\xi\|_{L^2(\Omega)}^2 \big).
\end{align*}
Noticing now that \eqref{eq:MVSDEMallDeriv} is a standard linear SDE (in $DZ$) with random (time-continuous) coefficients satisfying a one-sided Lipschitz condition, i.e. the SDE for $D Z$ is linear and satisfies \cite{imkeller2018differentiability}*{Assumption 2.4} with their $b(s,\omega),\sigma(s,\omega)$ set to zero, hence the moment estimate of \cite{imkeller2018differentiability}*{Theorem 2.5} applies with general $p\geq 2$ and yields 
\begin{align*} 
\nonumber 
\|D_\cdot Z \|_{\cS^p([0,T])}^p
& =
\bE\big[\, \sup_{0\leq t\leq T}|D_\cdot Z^n_t|^p\, \big]
\leq 
C \bE\big[\,  |\sigma(\cdot, Z_\cdot,\mu_\cdot)|^p\, \big]
\\
\nonumber 
& 
\leq 
C\Big(1+\|Z^n\|_{\cS^p([0,T])}^p + \sup_{0\leq t\leq T} W_2^p(\mu_t,\delta_0) \Big)
\\
& 
\leq C(1+\|\xi\|^p_{L^p(\Omega)})
 <\infty. 
\end{align*}
Thus, akin to Estimate \eqref{eq:aux: unifrom estimates}, we obtain from the above  inequality 
\begin{align}
\sup_{s\in[0,T]}\bE[\, \sup_{t\in[0,T]}  |D_s Z_t|^p \,  ] 
=
\sup_{s\in[0,T]} \|D_s Z \|_{\cS^p([0,T])}^p
\leq C(1+\|\xi\|^p_{L^p(\Omega)}).
\label{eq:auc:yetanotherrandomlabel-2ndtimearound}
\end{align}
Further, following the footsteps of \eqref{eq:aux:Hilbert} but this time in $p$-moment norms we have
\begin{align*}
\notag
\bE\Big[ \sup_{t\in [0,T]} \Big(\int_0^T |D_s Z_t|^2 ds\Big)^{\frac p2} \Big]
&
\leq
C \bE\Big[  \int_0^T  \big\{\sup_{t\in [0,T]} |D_s Z_t|^p  \big\}  ds  \Big]
\\ \notag
& 
\leq C 
\int_0^T \bE\Big[ \sup_{t\in [0,T]} |D_s Z_t|^p  \Big] ds
\\ \notag
& 
\leq C 
\sup_{s\in [0,T]} \bE\Big[ \sup_{t\in [0,T]} |D_s Z_t|^p \Big]
\\
& 
   \leq C(1+\|\xi\|^p_{L^p(\Omega)})
    <\infty, 
\end{align*}
where we firstly used Jensen's inequality and moved the supremum inside the integral, then used Fubini and dominated the integral by the supremum over the integration variable in a way that  \eqref{eq:auc:yetanotherrandomlabel-2ndtimearound} can be used. This shows \eqref{eq:aux:HilbertSpaceDominatoin}. 

Having shown that $Z\in \cS^p([0,T])$ and that $DZ$ has higher-order $p$-moments (in the sense of \eqref{eq:auc:yetanotherrandomlabel-2ndtimearound} and \eqref{eq:aux:HilbertSpaceDominatoin}) we conclude via \cite{nualart2006malliavin}*{Proposition 1.5.5} that $Z\in \bD^{1,p}(\cS^p)$ (and space interpolation that $Z\in \bD^{1,q}(\cS^q)$ for any $q\in [2,p]$).

\color{black}

\end{proof}

\section{Malliavin differentiability via the interacting particle system}
\label{Sec: IPS under Lipschitz}

The main goal of this section is to explore, in a didactic fashion, how much of  Theorem \ref{proposition:MallDiff-all-around} can be recovered under the interacting particle system approach. A somewhat close approach has been taken, for instance, in \cite{haji2021simple} and \cite{chen2024improved}. 

To simplify arguments we will work under the further restriction of full Lipschitz conditions on the MV-SDE's coefficients and thus abdicate the more general super-linear growth and one-sided Lipschitz assumption.

\subsection{The interacting and non-interacting particle system} 
We introduce the interacting particle system (IPS) associated to McKean-Vlasov SDE \eqref{Eq:General MVSDE}. Consider the system of SDEs for $ 
i=1,\cdots, N$:
\begin{align}
\label{eq:n+1PartSys-Xns}
    d X^{i}_t 
    = 
    b(t,X^{i}_t,\bar{\mu}^N_t)dt 
    + \sum_{l=1}^m\sigma^l(t,X^{i}_t,\bar{\mu}^N_t)dW^{l,i}_t,
    \quad 
    X^{i}_0=\xi^i,
\end{align}
where $\bar{\mu}^N_t(dy)=\frac1N \sum_{k=1}^N \delta_{X^i_t}(dy)$ and for  $l =1,\ldots,m$, $\{W^{l,i}\}_{i = 1,\cdots,N}$, are independent 1-dimensional Brownian motions and $\{\xi^i\}_{i = 1,\cdots,N}$ are i.i.d.~copies of $\xi$; the $(W^{l,i},\xi^i)_i$ in \eqref{eq:n+1PartSys-Xns} are independent of $W^l,\xi$ in \eqref{Eq:General MVSDE} (and in fact, live in different probability spaces). We write $W^i = (W^{1,i},\ldots, W^{m,i})$ to be the corresponding $m-$dimensional Brownian motions for $i=1,\ldots,N$. The dependence on the empirical distributions in the coefficients introduces non-linearity into the system in the form of self-interaction; hence we refer to the above set of equations as an \emph{interacting particle system} (IPS). 

Since \eqref{Eq:General MVSDE} and \eqref{eq:n+1PartSys-Xns} live in different probability spaces, we construct an auxiliary \emph{non-interacting particle system} (non-IPS) as living in the same probability space as \eqref{eq:n+1PartSys-Xns}. For $i=1,\cdots, N$ ,
\begin{align}
\label{eq:n+1nonintPartSys-Xns}
    d Z^i_t 
    = 
    b(t,Z^i_t,\mu^i_t)dt 
    +\sum_{l=1}^m \sigma^l(t,Z^i_t,\mu^i_t)dW^{l,i}_t,
    \qquad 
    Z^i_0=\xi^i, 
\end{align}
where $\mu^i$ is defined as the law of $Z^i$. In this case, the $\{Z^i\}_{i = 1,\cdots,N}$ are independent of each other, since the $(W^i,\xi^i)_i$ are all i.i.d.  and $\mu^i = \mu^j =\mu $ $\forall \,i,j = 1,\cdots, N$ where $\mu$ denotes the law of the McKean-Vlasov SDE \eqref{Eq:General MVSDE}. In essence, \eqref{eq:n+1nonintPartSys-Xns} is a decoupled system of $N$ copies of \eqref{Eq:General MVSDE}. 
From direct inspection of \eqref{eq:n+1nonintPartSys-Xns}, we have the following lemma regarding the cross-Malliavin derivatives $D^j Z^i$ for $i\neq j$.
\begin{lemma}
\label{lemma:djzi}
Assume \eqref{eq:n+1nonintPartSys-Xns} is well-posed. 
Then, the cross-Malliavin derivatives of the solution $\{Z^i\}_{i = 1,\cdots,N}$ to \eqref{eq:n+1nonintPartSys-Xns} are all zero. That is, 
\begin{align*}
    D^j_s Z^i_t=0
    \quad  \textrm{ for any }
    \quad j\neq i, 1\leq i,j\leq N, \quad s,t\in [0,T]. 
\end{align*}    
\end{lemma}
\begin{proof}
    One sees that $Z^i$ acts independently of any other Brownian motion $W^j$ for $j\neq i$. The result follows immediately from the definition of the Malliavin derivative.
\end{proof}
\subsubsection{Preliminaries}
 In this section, we work under stronger assumptions requiring $b$ and $\sigma$ to be globally space-measure Lipschitz. Formally we state the framework as follows:

\begin{assumption}
\label{ass:MKSDE-MainExistTheo2}
Let $b:[0,T] \times \bR^d \times\cP_2(\bR^d) \to \bR^d$ and for $l=1,\ldots,m$, $\sigma^l:[0,T] \times \bR^d \times \cP_2(\bR^d) \to \bR^{d}$ be progressively measurable deterministic maps and $\exists L>0$ such that: 
\begin{enumerate}
\item For some $p\geq 2$, $\xi^i \in L^p(\Omega, \cF_0, \bP; \bR^d)$ for $i=1,\cdots,N$, 
\item $\sigma^l$ is continuous in time and Lipschitz in space-measure $\forall t \in[0,T]$, $\forall x, x'\in \bR^d$ and $\forall \mu, \mu' \in \mathcal{P}_2(\bR^d)$ we have
\begin{align*}
    |\sigma^l(t,  x, \mu) - \sigma^l(t , x', \mu')| \leq L\Big( |x-x'| +W_2(\mu, \mu')\Big).
\end{align*}
\item $b$ is continuous in time and Lipschitz in space-measure $\forall t \in[0,T]$, $\forall x, x'\in \bR^d$ and $\forall \mu, \mu' \in \mathcal{P}_2(\bR^d)$ we have
\begin{align*}
    |b(t,  x, \mu) - b(t , x', \mu')| \leq L\Big( |x-x'| +W_2(\mu, \mu')\Big).
\end{align*}
\end{enumerate}
\end{assumption}


A quick inspection of \eqref{eq:n+1PartSys-Xns} highlights that this system of equations can be seen as a system in $(\bR^d)^N$ as opposed to $N$ equations valued in $\bR^d$. The former has a few advantages and to formalize it we introduce the notion of  an \textit{empirical projection} introduced in \cite{book:CarmonaDelarue2018a}*{Definition 5.34}.
\begin{definition}[Empirical projection of a map]
\label{def:Auxiliary-uN-for-EmpirialTrick}
Given $u: \cP_2(\bR^d) \to \bR^d$ and $N\in \bN$, define the empirical projection $u^N$ of $u$ via $u^N: (\bR^d)^N \to \bR^d$, such that
\begin{align*}
    u^N(x^1,\dots, x^N) := u \big(\bar{\mu}^N\big),
\quad \text{with}\quad
\bar{\mu}^N (\dd y) := \frac{1}{N}\sum\limits_{l=1}^N \delta_{x^l}\, (\dd y),
\end{align*}
for $x^l\in \bR^d, l=1,\dots,N$. \color{black}
\end{definition}

We can use a similar notion of mapping points onto empirical projections to express the interacting particle system as a high-dimensional SDE. 
Hence we can interpret system \eqref{eq:n+1PartSys-Xns} as a system in $(\bR^d)^N$ with $B^N:[0,T]\times (\bR^d)^N\to (\bR^d)^N$ and $\Sigma^N:[0,T]\times (\bR^d)^N\to \bR^{mN\times dN}$\color{black}
,  
\begin{align}
\label{eq:VeryBigSDE}
d\mathbf{X}_t=B^N(t,\mathbf{X}_t)dt+\Sigma^N(t,\mathbf{X}_t)d\mathbf{W}_t, \quad \mathbf{X}_0=\boldsymbol \xi,
\end{align}
for $\mathbf{X}=(X^1,\cdots,X^N)$, $\mathbf{W}=(W^1,\cdots,W^N)$ and $\boldsymbol{\xi}=(\xi^1,\cdots,\xi^N)$ where for $t\in[0,T]$ and  $\mathbf{x} = (x^1,\cdots, x^N) \in \bR^{dN}$, $x^i\in\bR^d$, $i=1,\cdots,N$ we have 
\begin{align*}
    B^N(t,{\bf x})
    &=\big( b^N_1(t,{\bf x}),\cdots,b^N_N(t,{\bf x}) \big)
    \\
    &
    :=\big( b(t,x^1,\bar{\mu}^N),\cdots,b(t,x^N,\bar{\mu}^N) \big)
    \\
    \Sigma^N(t,{\bf x}) 
    & =\textrm{diag}\big(\sigma^N_1(t,{\bf x}),\cdots,\sigma^N_N(t,{\bf x})\big)
    \\
    & 
    :=\textrm{diag}\big(\sigma(t,x^1,\bar{\mu}^N),\cdots,\sigma(t,x^N,\bar{\mu}^N)\big),
\end{align*}
with the relation between ${\bf x}$ and $\bar{\mu}^N$ being as highlighted in Definition \ref{def:Auxiliary-uN-for-EmpirialTrick}. The next result shows that from the Lipschitz properties of $b,\sigma$ in space and measure, one can show that ${\bf x}\mapsto B^N(\cdot,{\bf x}),\Sigma^N(\cdot, {\bf x})$ are uniformly Lipschitz (uniformly in time).
\begin{lemma}
\label{lemma:PropertiesofBandSigma2}
Under Assumption \ref{ass:MKSDE-MainExistTheo}, the maps $B^N$ and $\Sigma^N$ in \eqref{eq:VeryBigSDE} are globally Lipschitz in their spatial variables. 
\end{lemma}
\begin{proof}
Let $\mathbf{x} = (x^1,\cdots, x^N),\mathbf{y} = (y^1,\cdots, y^N)\in \bR^{dN}$, for $x^i,y^i\in\bR^d$, $i=1,\cdots,N$ then
\begin{align*}
|B^ N(t,
{\bf x})
& -B^N(t,{\bf y})|^2
=
\sum_{k=1}^N 
\big| b(t,x^k,\bar{\mu}^N({\bf x})-b(t,y^k,\bar{\mu}^N({\bf y}) \big|^2
\\
&
\leq
L^2\sum_{k=1}^N
\Big(|x^k-y^k| +W_2(\bar{\mu}^N({\bf x}),\bar{\mu}^N({\bf y})\Big)^2
\\
&
\leq
2L^2\sum_{k=1}^N
|x^k-y^k|^2 +W_2^2(\bar{\mu}^N({\bf x}),\bar{\mu}^N({\bf y}))
\leq
4L^2|{\bf x}-{\bf y}|^2, 
\end{align*}
    with the final inequality arising from the fact
    \begin{align}
    \label{eq:wassersteinempirical}
        W_2\left(\frac{1}{N}\sum_{i=1}^N \delta_{x^i},\frac{1}{N}\sum_{i=1}^N \delta_{y^i}\right)\leq \left(\frac{1}{N}\sum_{i=1}^N|x^i-y^i|^2\right)^{\frac{1}{2}} 
        = \frac{1}{\sqrt{N}}|\mathbf{x}-\mathbf{y}|.
    \end{align}
    The proof is similar for $\bR^{dN} \ni {\mathbf x} \mapsto \Sigma^N({\mathbf x})\in \bR^{mN\times dN}$. 
\end{proof}

\subsubsection{Classical results}

We briefly recall classical results involving the relationship between systems described above. Well-posedness follows from classic literature, while the second result is classically known as Propagation of Chaos (PoC), which ascertains convergence of $X^i$ to $Z^i$ as $N\to \infty$ (as respective laws) \cite{Sznitman1991}.
\begin{proposition}[Well-posedness and Propagation of Chaos]
\label{lemma:PoCforBSigmaSystem}
Let Assumption \ref{ass:MKSDE-MainExistTheo2} hold. Then, the solutions to the systems \eqref{eq:n+1PartSys-Xns} and \eqref{eq:n+1nonintPartSys-Xns}, given by $\{X_t^{i}\}_{i=1,\cdots, N}$ and $\{Z_t^i\}_{i=1,\cdots, N}$ respectively are well-posed, unique and square integrable. 
It holds that
\begin{align}
\label{IPSbound}
\sup_{N\in\bN} \,\max_{1\leq i\leq N} 
\Big\{\bE\big[\sup_{0\leq t \leq T} |Z^i_t|^2\big] 
+
\bE\big[\sup_{0\leq t \leq T} |X^i_t|^2\big]  \Big\}
\leq C (1+\bE [|\xi^\cdot |^2 ])e^{CT}<\infty,
\end{align}
where the involved constant $C$ depends on $d,m,L$ and the quantity
\begin{align*}
    \int_0^T |b(t, 0, \delta_0)| dt + \int_0^T |\sigma(t, 0, \delta_0)|^2 dt,
\end{align*}
but independent of $N$. Moreover, we have for any $i=1,\cdots,N$, 
    \begin{align}
    \label{eq:poc}
    \lim_{N\to \infty} \sup_{0\leq t\leq T} \bE[W_2^2(\mu^i_t,\bar \mu^N_t)] =0
    \hspace{0.2cm}  \textrm{and} 
        \lim_{N\to \infty}\max_{1\leq i\leq N}
        \bE \Big[\,\sup_{0\leq t\leq T}\left|X_t^{i}-Z_t^i\right|^2 \Big] = 0.    
    \end{align}
\end{proposition}

\begin{proof}
Under Lipschitz conditions this result is classical. 
Well-posedness of \eqref{eq:n+1nonintPartSys-Xns} follows directly from Theorem \ref{theo:Wellposendess} as it is a non-interacting particle system (thus well-posedness of the initial McKean-Vlasov equation suffices). 

As for system \eqref{eq:n+1PartSys-Xns}, Lemma \ref{lemma:PropertiesofBandSigma2} ensures the coefficients are uniformly Lipschitz (as maps in $(\bR^d)^N$) and thus well-posedness (for fixed $N$) follows from general SDE theory \cite{gyongykrylov1980SDEMartingales}*{Theorem 1}. 
One can conclude the uniform in $N$ estimates of \eqref{IPSbound} for $\{X^i\}_i$ by mimicking the arguments used in the proof of Theorem \ref{proposition:MallDiff-all-around}. 

The convergence results of \eqref{eq:poc} follows from \cite{carmona2015LecNotesBook}*{Lemma 1.9 and Theorem 1.10}. 
\end{proof}

\subsubsection{A primer on Lions derivatives}
\label{sec:MeasureDerivatives}

To consider the calculus for the mean-field setting, one requires to build a suitable differentiation operator on $2$-Wasserstein space. Among the numerous notions of differentiability of a function $u$ defined over the $\cP_2(\bR^d)$, we try to follow the approach introduced by Lions in his lectures at Coll\`ege de France. A comprehensive collection of recent results was done in the joint monographs of Carmona and Delarue \cite{book:CarmonaDelarue2018a}, \cite{book:CarmonaDelarue2018b}. In line with the construction we assume our probability space to be an atomless Polish space \cite{book:CarmonaDelarue2018a}*{Chapter 5}.

We consider a canonical lifting of the function $u:\cP_2(\bR^d) \to \bR^d$ to $\tilde{u}: L^2(\Omega,\cF,\bP;\bR^d) \ni X \to \tilde u (X) = u(Law(X)) \in \bR^d$. We can say that $u$ is L-differentiable at $\mu$, if $\tilde u$ is Fréchet differentiable at some $X$, such that $\mu = \bP \circ X^{(-1)}$. Denoting the gradient by $D\tilde u$ and using a Hilbert structure of the $L^2$ space, we can identify $D\tilde u$ as an element of $L^2$. It has been shown that $D\tilde u$ is a $\sigma(X)$-measurable random variable and given by the function $Du(\mu)( \cdot) : \bR^d \to \bR^d $, depending on the law of $X$ and satisfying $Du (\mu)( \cdot) \in L^2(\bR^d, \cB(\bR^d),\mu; \bR^d)$. Hereinafter the L-derivative of $u$ at $\mu$ is the map $\partial_\mu u(\mu)(\cdot): \bR^d \ni v \to \partial_\mu u(\mu)(v) \in \bR^d$, satisfying $D\tilde u(X) = \partial_\mu u(\mu)(X)$.
We always denote $\partial_\mu u$ as the version of the L-derivative that is continuous in product topology of all components of $u$. 

Definition \ref{def:Auxiliary-uN-for-EmpirialTrick} relates the spatial derivatives of $u^N$ with the Lions derivative of the measure function $u$. Such is stated next; see also \cite{book:CarmonaDelarue2018a}*{Proposition 5.35 (p.399)}.  
\begin{proposition}
\label{prop:DerivativeRelations-Space-2-Lions}
Let $u: \cP_2(\bR^d) \to \bR^d$ be a continuously $L$-differentiable map, then, for any $N>1$, the empirical projection $u^N$ is differentiable in $(\bR^d)^N$ and for all $x^1,\cdots,x^N\in\bR^d$ we have the following relation:
\begin{align*}
    \partial_{x^j}u^N(x^1, \dots, x^N) &= \frac{1}{N}\: \partial_\mu u\Big(\frac{1}{N}\sum_{l=1}^N \delta_{x^l}\Big)(x^j).
\end{align*}
\end{proposition}

\subsection{Exploring Malliavin differentiability via interacting particle system limits}
\label{sec:Mall diff under diff assumptions}

The novelty in this section lies not within the results, as these are implied directly by those in our Section \ref{section:Mall Diff superlinear growth} or \cite{Crisan2018}, but in the proof methodology via limits of interacting particle systems. 

\begin{assumption}
\label{Assump:All is Lipschitz } 
Let Assumption \ref{ass:MKSDE-MainExistTheo2} hold. 
\begin{enumerate}
    \item The functions $b,\sigma^l$, $l=1,\ldots,m$ are continuously differentiable in their spatial variables and their spatial derivative maps are continuous in time. 
     Further, the maps
    $(\nabla_x b)$, $(\nabla_x \sigma^l)(t, x, \mu)$ are uniformly bounded for $l=1,\ldots,m$ (over all variables $(t,x,\mu) \in [0,T] \times\bR^d \times \cP_2(\bR^d)$).
    
    \item  For any $t\in[0,T]$ the maps $\mu\mapsto b(t,x,\mu)$ and $\mu\mapsto \sigma^l(t,x,\mu)$, $l=1,\ldots,m$, are $\bP$-a.s.~continuous in topology, induced by the Wasserstein metric and $L$-differentiable $\bP$-a.s. at every $\mu \in \cP_2(\bR^d).$ Moreover, $\partial_\mu b(t,x,\mu)(v)$ and $\partial_\mu \sigma^l(t,x,\mu)(v)$ have $\mu$-versions such that $\partial_\mu b(t,x,\mu)(v)$ and $\partial_\mu \sigma^l(t,x,\mu)(v)$ are $\bP$-a.s.~joint-continuous at every quadruple $(t,x,\mu,v)$ with $(t,x,\mu) \in [0,T] \times\bR^d \times \cP_2(\bR^d),~v \in \Supp(\mu)$ and uniformly bounded (in all variables). \color{black}
\end{enumerate}

\end{assumption}

Focusing on the technique of limits of particle systems, we establish the following result that covers the Malliavin differentiability of the particle system, the Propagation of Chaos result and how to transfer the Malliavin regularity to the limiting McKean-Vlasov SDE.
\begin{proposition}
\label{malliavinIPS}
Let Assumption \ref{Assump:All is Lipschitz } hold.
Then the solution $\{Z^i\}_{i = 1,\cdots,N}$ to \eqref{eq:n+1nonintPartSys-Xns} and the solution $\{X^i\}_{i = 1,\cdots,N}$ to \eqref{eq:n+1PartSys-Xns} are Malliavin differentiable with Malliavin derivatives $\{D^j Z^i\}_{i,j=1,\cdots,N}$ and $\{D^j X^i\}_{i,j=1,\cdots,N}$ respectively. \\

For $\{D^j Z^i\}_{i,j=1,\cdots,N}$ we have that:
\begin{itemize}
    \item  For any $j\neq i$, $1\leq i,j\leq N$ $s,t\in [0,T]$ that $D^j_s Z^i_t=0$ (Lemma \ref{lemma:djzi}).
    \item  When $j=i$ then $D^i Z^i$ satisfies for $0\leq s\leq t \leq T$
\begin{align}
\label{Zmd}
D^i_s Z^i_t
= \sigma(s,Z^i_s, \mu^i_s) 
&
+ \int_s^t( \nabla_x b)(r, Z^i_r, \mu^i_r) D^i_s Z^i_r dr \notag
\\
&
+ \sum_{l=1}^m\int_s^t (\nabla_x \sigma^l)(r,Z^i_r, \mu^i_r) D^i_s Z^i_r dW^{l,i}_r.
\end{align}
If $s >t$ then $D^i_s Z^i_t=0$ $\bP$-almost surely. 

\item Moreover, for some $C>0$ dependent on $T$ and $L$ but not on $N$, (c.f.~Theorem \ref{proposition:MallDiff-all-around})
\begin{align*}
\sup_{0\leq s\leq T} \|D^i_s Z^i \|_{\cS^2([0,T])}^2
&
\leq 
\sup_{0\leq s\leq T} \bE\Big[\sup_{0\leq t \leq T}|D^i_s Z^i_t|^2\Big]
\\
&
\leq C(1+\|Z^i\|_{\cS^2([0,T])}^2)
\leq C(1+\|\xi^\cdot\|_{L^2(\Omega)}^2)
 <\infty.
\end{align*}
\end{itemize}

For $\{D^j X^i\}_{i,j=1,\cdots,N}$, the Malliavin derivative of \eqref{eq:n+1PartSys-Xns}, we have that: 
\begin{itemize}
    \item If $s >t $ then $D^j_s X^i_t=0$ $\bP$-almost surely for any $i,j=1,\cdots, N$.
    \item If $0 \leq s\leq t $ then $D^j_s X^i_t$ satisfies
    \begin{align}
\label{malderiv}
    D_s^j \notag
    X_t^{i} 
    &= \sigma
    (s,X_s^{i},\bar{\mu}_s^N)  \1_{i=j} 
    \\
    &
    +\int_{s}^{t} \bigg\{(\nabla_x b)(r,X_r^{i},\bar{\mu}_r^N)D_s^jX_r^{i} \bigg.
    \\
    &
    \hspace{2cm}+\left. \frac{1}{N}\sum_{k=1}^{N}(\partial_\mu b)(r,X_r^{i},\bar{\mu}_r^N)(X^{k}_r)D_s^j X_r^{k} \right\}\, dr \notag
    \\
    \notag 
    &
    +\sum_{l=1}^m\int_{s}^{t} \bigg\{(\nabla_x \sigma^l)(r,X_r^i,\bar{\mu}_r^N)D_s^jX_r^{i} \bigg.
    \\
    &
    \left.
    \hspace{2cm}
    +\frac{1}{N}\sum_{k=1}^{N}(\partial_\mu \sigma^l)(r,X_r^{i},\bar{\mu}_r^N)(X^{k}_r)D_s^j X_r^{k} \right\}\, dW_r^i.
\end{align}
    
    \item Moreover, there exists a constant $C>0$ depending on $T$ and $L$ but not on $N$ such that
\begin{align*}
        \sup_{0\leq s\leq T}  
        \bE \left[\,
         \sup_{0 \leq t\leq T} |D_s^j X_t^{i}|^2 \, \right] 
        &
        \leq C 
        \Big( 
        \1_{i=j}
        + \frac{1}{N}   
               \Big).
\end{align*}\color{black}
\end{itemize}

Finally, for any $s,t \in [0,T]$, 
    \begin{itemize}
        \item  $D^j_s X^i_t \to 0$ as $N\to \infty$ for $j\neq i$ in  $L^2(\Omega)$ and almost surely, and
        \item  $D^i_s X^i_t \to D^i_s Z^i_t $ as $N\to \infty$ in $L^2(\Omega)$.
    \end{itemize} 
\end{proposition}
\begin{proof}
Rewrite the IPS \eqref{eq:n+1PartSys-Xns} in integral form: for $i = 1,\ldots,N$,
\begin{align}
\label{eq:SDE}
X_t^{i} 
&
= \xi^i + \int_{0}^{t}b(r,X_r^{i},\bar{\mu}_r^N) dr + \sum_{l=1}^m\int_{0}^{t}\sigma^l (r,X_r^{i},\bar{\mu}_r^N) dW_r^{l,i}, 
\\ \notag
\label{eq:SDE-Empirical Projs}
&
= \xi^i 
+ \int_{0}^{t} b^N_i(r,X_r^{1},\cdots, X_r^{N}) dr 
\\
& \hspace{3cm}
+ \sum_{l=1}^m\int_{0}^{t}\sigma^{l,N}_i (r,X_r^{1},\cdots, X_r^{N}) dW_r^{l,i}, 
\end{align}
where \eqref{eq:SDE-Empirical Projs} uses the empirical projection representation (Definition \ref{def:Auxiliary-uN-for-EmpirialTrick}) of the coefficients in \eqref{eq:SDE}; that is  $b^N_i(t,X^1,\cdots,X^N):= b(t,X^{i},\bar{\mu}^N) $ and similarly for $\sigma^{l,N}_i$. 
In view of Assumption \ref{Assump:All is Lipschitz }, it is easy to conclude that the coefficients $b_i^N,\sigma_i^{l,N}$ of \eqref{eq:SDE} are Lipschitz continuous (via Lemma \ref{lemma:PropertiesofBandSigma2}) and also differentiable in their variables (via Proposition \ref{prop:DerivativeRelations-Space-2-Lions}). From the results in \cite{nualart2006malliavin} or \cite{imkeller2018differentiability} for classical SDEs we conclude immediately that the  Malliavin derivatives of $X^i$ exist, are unique and are square-integrable. Furthermore, if $s >t $ then $D^j_s X^i_t=0$ $\bP$-a.s. for any $i,j=1,\cdots, N$.

From application of the chain rule, the Malliavin derivative is written for $0\leq s\leq t\leq T < \infty$ as 
\begin{align*}
    D_s^j X_t^{i} 
    = 
    \sigma^N_i(s,X^1_s,\cdots,X^N_s)  \1_{i=j}
    & 
    +\int_s^t 
    \sum_{k=1}^N( \partial_{x^k}  b^N_i)(r,X^1_r,\cdots,X^N_r)D^j_s X_r^k \dd r 
    \notag
    \\  
    & \hspace{-1.4cm}
    +\sum_{l=1}^m\int_s^t 
    \sum_{k=1}^N (\partial_{x^k}  \sigma^{l,N}_i)(r,X^1_r,\cdots,X^N_r)D^j_s X_r^k  
    \dd W_r^{l,i} .
\end{align*}
Exploiting Proposition \ref{prop:DerivativeRelations-Space-2-Lions} and reverting the empirical projection maps to their original form, we rewrite this as 
\begin{align*}
D_s^j \notag
    X_t^{i} 
    &= \sigma
    (s,X_s^{i},\bar{\mu}_s^N)  \1_{i=j} 
    \\
    &
    +\int_{s}^{t} \bigg\{(\nabla_x b)(r,X_r^{i},\bar{\mu}_r^N)D_s^jX_r^{i} \bigg.
    \\
    &
    \hspace{2cm}+\left. \frac{1}{N}\sum_{k=1}^{N}(\partial_\mu b)(r,X_r^{i},\bar{\mu}_r^N)(X^{k}_r)D_s^j X_r^{k} \right\}\, dr \notag
    \\
    \notag 
    &
    +\sum_{l=1}^m\int_{s}^{t} \bigg\{(\nabla_x \sigma^l)(r,X_r^i,\bar{\mu}_r^N)D_s^jX_r^{i} \bigg.
    \\
    &
    \left.
    \hspace{2cm}
    +\frac{1}{N}\sum_{k=1}^{N}(\partial_\mu \sigma^l)(r,X_r^{i},\bar{\mu}_r^N)(X^{k}_r)D_s^j X_r^{k} \right\}\, dW_r^i.
\end{align*}
noting that the derivatives $\partial_{x^i} b^N_i$ and $\partial_{x^i} \sigma^{l,N}_i$ produce two components as opposed to the one produced by the cross-terms $\partial_{x^k} b^N_i$ and $\partial_{x^k} \sigma^{l,N}_i$, $k\neq i$. \\

\textit{Step 1. Preliminary manipulations.} 
Applying Itô's formula to \eqref{malderiv},
\begin{align}
\label{eq:itomalderiv}
|&D_s^j 
X_t^{i}|^2=
|\sigma(s,X_s^{i},\bar{\mu}_s^N)|^2  \1_{i=j} 
\\
&
+
\int_{s}^{t}2D_s^j X_r^{i}:\bigg\{(\nabla_x b )(r,X_r^{i},\bar{\mu}_r^N)D_s^jX_r^{i} \bigg.\\
& \hspace{3cm}
+\bigg.\frac{1}{N}\sum_{k=1}^{N}(\partial_\mu b)(r,X_r^{i},\bar{\mu}_r^N)(X^{k}_r)D_s^j X_r^{k}\bigg\} 
\notag \\
\notag
&
+ \sum_{l=1}^m\Big| (\nabla_x \sigma^l)(r,X_r^{i},\bar{\mu}_r^N)D_s^jX_r^{i} +\frac{1}{N}\sum_{k=1}^{N}(\partial_\mu \sigma^l)(r,X_r^{i},\bar{\mu}_r^N)(X_r^{k})D_s^j X_r^{k} \Big| ^2 \, dr 
\\
&
+ M_N(s,t), 
\end{align}
where $M_N$ is a Hilbert space-valued local martingale term. In fact, by standard SDE well-posedness theory, e.g. \cite{gyongykrylov1980SDEMartingales}*{Theorem 1}, for any $s \in [0,T]$, a $L^2$-integrable solution to Equation \eqref{eq:itomalderiv} exists, allowing us to conclude $M_N$ is a proper martingale for finite $N$.\footnote{We note that \cite{gyongykrylov1980SDEMartingales}*{Theorem 1} provides wellposedness and moment estimates for the SDE, but critically, such estimates are \textit{dependent on $N$} and explode as $N\nearrow +\infty$. Nonetheless, for any fixed $N$ the estimates suffice to ensure that $\bE[M_N(s,t)]=0$ for any $s,t$. Obtaining moment estimates uniformly in $N$ is done in a subsequent step of the proof. This is the exact same as in the proof of Proposition \ref{lemma:PoCforBSigmaSystem}.}.

Let $L$ be the Lipschitz constant which bounds the spatial and measure derivatives of $b$ and $\sigma^l$, $l=1,\ldots,m$. Temporarily fix $s \in [0,T]$. We have
\begin{align}
\label{expectationitomalderiv}
  \bE\big[& \sup_{0\leq t \leq T}\big|D_s^j  X_t^{i} \big|^2\big] 
  \nonumber
  \\
    & \leq 
     \bE\big[|
    \sigma(s,X_s^{i},\bar{\mu}_s^N)  \1_{i=j}|^2\big] +\bE[\sup_{0\leq t \leq T}|M_N(t,s)|] \notag
    \\
    &
    +
    \bE\bigg[ \sup_{0\leq t \leq T}\int_s^t 2L |D_s^j X_r^{i}|^2 + 
     2mL^2 |D_s^j X_r^{i}|^2
     \notag 
    \\
    & 
    \hspace{2cm}+\frac{2L|D_s^j X_r^{i}|}{N}\sum_{k=1}^{N}|D_s^j X_r^{k}| + 2mL^2\Big|\frac{1}{N}\sum_{k=1}^{N}|D_s^j X_r^{k}| \Big|^2 \, \dd r \bigg] \notag
\\
&
    \leq 
     \bE\big[|\sigma(s,X_s^{i},\bar{\mu}_s^N)  \1_{i=j}|^2\big] 
     + \kappa \bE\bigg[\sup_{0\leq t \leq T} 
     \int_s^t 
     |D_s^j X_r^{i}|^2 
     + \frac{1}{N}\sum_{k=1}^{N} |D_s^j X_r^{k}|^2 
     \dd r
 \bigg], 
\end{align}
for some $\kappa >0$, by applying the Young and Jensen inequalities, where we used the BDG inequality to control the martingale term. Further, one can bound $\sup_{0\leq s\leq T}\bE[\,|\sigma(s,X_s^{i},\bar{\mu}_s^N)|^2]$ uniformly in $N$, justified by well-posedness and the uniform in $N$ moment bounds of the SDE-IPS (and also using the continuity assumption on $t\mapsto \sigma(t,0,\delta_0)$). That is, by linear growth of $\sigma$, properties of the Wasserstein metric and Proposition \ref{lemma:PoCforBSigmaSystem}, 
\begin{align}
\label{upperboundsigma}
\notag 
\sup_{0\leq s\leq T}\bE\big[ \, 
& |\sigma(s,X_s^{i}, \bar{\mu}_s^N)|^2\big]
\\ 
& \leq 
C 
\Big(  1 + \sup_{0\leq s\leq T}\bE\big[\,  |X_s^{i}|^2 \big] + \frac1N \sum_{k=1}^N\sup_{0\leq s\leq T} \bE\big[\, |X_s^{k}|^2 \big] \Big) \notag
\\ 
& 
\leq  
{C} (1+\bE [|\xi^\cdot |^2 ])e^{{C}T} =: \alpha,
\end{align}
for some $C>0$ independent of $N$, observing that the $\{\xi^j\}_j$ are i.i.d.  
\\

\textit{Step 2. Controlling the empirical mean of the Malliavin derivative.} We now aim to gain control over the quantity $\frac{1}{N}\sum_{i=1}^{N}D_s^jX_r^{i}$.  Averaging \eqref{expectationitomalderiv} over the index $i$ and once more using that the $\{\xi^j\}_j$ are i.i.d., we have 
\begin{align}
    \notag 
    \bE\Big[\frac{1}{N}\sum_{i=1}^{N}
    & 
    \sup_{0\leq t\leq T}|D_s^j X_t^{i}|^2\Big ] 
    \\
    \notag 
    & \leq \frac{\alpha}{N} 
    +
    \bE\Bigg[ \frac{\kappa}{N}\sum_{i=1}^N 
    \int_s^T\Big\{|D_s^j X_r^{i}|^2 +\frac{1}{N}\sum_{m=1}^N |D_s^j X_r^{m}|^2 \Big\}\dd r\Bigg] \notag
    \\
    &
    \label{eq: exp sup mal deriv square}
    \leq
    \frac{\alpha}{N}
    +   2\kappa\bE  \Bigg[\frac{1}{N}\int_s^T\sum_{i=1}^N|D_s^j X_r^i|^2\dd r\Bigg].  
   \end{align} 
Taking the supremum inside the expectation means we are not able to apply Grönwall's inequality directly. However, we have that
 \begin{align*}
     \bE\left[\frac{1}{N}\sum_{i=1}^{N}|D_s^j X_r^{i}|^2\right] 
     \leq 
     \frac{\alpha}{N}+2\kappa\bE\Bigg[\frac{1}{N}\int_s^r\sum_{i=1}^N|D_s^j X_\rho^i|^2\dd \rho\Bigg].
 \end{align*}
Applying Grönwall's inequality yields 
\begin{align} 
\label{eq:aux mean moment estimate v1}
   \bE\left[\frac{1}{N}\sum_{i=1}^{N}|D_s^j X_r^{i}|^2\right]
    & \leq
    \frac{\alpha}{N} \sup_{0\leq r,s\leq T}\left(1+ e^{2\kappa(r-s)}(e^{-2\kappa s}-e^{-2\kappa r})\right)
    \leq  
    \frac{\alpha}{N}e^{2\kappa T}. 
\end{align}
Hence, substituting \eqref{eq:aux mean moment estimate v1} back into \eqref{eq: exp sup mal deriv square}, we get
\begin{align*}
    \sup_{0\leq s\leq T}\bE\left[\frac{1}{N}\sum_{i=1}^{N}\sup_{0\leq t\leq T}|D_s^j X_t^{i}|^2\right] \leq \frac{2\alpha}{N}(1+\kappa Te^{2\kappa T}):=\Psi_N.
\end{align*}
It is immediate to see that $\Psi_N$ is uniformly bounded over $N$, hence
\begin{align*}
\lim_{N\rightarrow\infty} \sup_{0\leq s \leq T}\bE\left[ \sup_{0\leq t \leq T}\frac{1}{N} \sum_{i=1}^{N}|D_s^j X_t^{i}|^2\right] = 0
\qquad \textrm{for any }\quad j=1,\cdots,N.
\end{align*} 
Using Jensen's inequality and \eqref{eq:aux mean moment estimate v1}, observe that
\begin{align}
\label{eq:convergenceofmeanmallderiv}
\sup_{0\leq s \leq T}\bE \Bigg[ \sup_{0\leq t \leq T}\frac{1}{N}&\sum_{i=1}^N |D_s^j X_t^{i}|^2\Bigg]
   \leq \Psi_N \to 0 \quad \textrm{as } N\to+\infty
   \\
   & \Rightarrow \quad \notag
   \lim_{N\rightarrow\infty }\frac{1}{N}\sum_{i=1}^{N}|D_s^j X_t^{k}| = 0 \,\,\bP\textrm{-a.s.} 
\end{align} 
\color{black}
\textit{Step 3. Convergence of the individual terms.} Injecting  the bound $\Psi_N$ into \eqref{expectationitomalderiv}, we obtain 
\begin{align*}
    \sup_{0\leq s \leq T}\bE\left[ \sup_{0\leq t \leq T}\big|D_s^j X_t^{i} \big|^2\right] \nonumber
     \leq
     \alpha\1_{i=j} 
     + \kappa T \Psi_N + \kappa\int_0^T \bE\big[ \big|D_s^j X_t^{i} \big|^2\big] dr.
\end{align*}
Applying Grönwall's inequality again, we bound 
\begin{align*}
    \bE \left[ |D_s^j X_t^{i}|^2\right]   
    &
    \leq\bigg( \alpha\1_{i=j}+\kappa T  \Psi_N \bigg)\notag
    +\kappa \int_s^t e^{\kappa(t-s-r)}\bigg( \alpha\1_{i=j}+ \kappa T\Psi_N \bigg) dr\notag
    \\
    & \Rightarrow 
  \bE \big [\, |D_s^j X_t^{i}|^2 \, \big] \leq e^{\kappa T}\bigg( \alpha\1_{i=j}+ \kappa T\Psi_N \bigg).
\end{align*}

Hence,
\begin{align}
\label{eq:gronwallindividualparticle}
    \sup_{0\leq s \leq T}\bE \left[ \sup_{0\leq t \leq T}|D_s^j X_t^{i}|^2\right]   
    &
    \leq\bigg( \alpha\1_{i=j}+\kappa T  \Psi_N \bigg)\bigg(1+ \kappa Te^{\kappa T}\bigg).
\end{align}
Hence, applying an identical argument to Equation \eqref{eq:aux:Hilbert}
we have:
\begin{align*}
    \sup_{N\in \bN}\max_{1\leq i,j\leq N}\bE\left[\sup_{0\leq t \leq T}\int_0^T|D_s^j X_t^{i}|^2 \dd s\right]<\infty.
\end{align*}
In particular, we have that $X_t^{i} \in \bD^{1,2}(\cS^2)$ uniformly in $N$.
 
 Aside, we obtain that $D_s^j X_t^{i} \rightarrow 0$ in $\bD^{1,2}(\cS^2) $ for $i\neq j$ as the size of interacting particle system  $N \rightarrow \infty $. This is in line with Lemma \ref{lemma:djzi}, since referring back to our non-IPS analogy, particles essentially become more conditionally independent as the particle system size gets larger.
 
 We can now apply \cite{nualart2006malliavin}*{Lemma 1.2.3}. Let $\{Z^i\}_{i=1,\cdots,N}$ denote the solution to the non-IPS \eqref{eq:n+1nonintPartSys-Xns}. By Proposition \ref{lemma:PoCforBSigmaSystem}, we have  
\begin{align*}
   \lim_{N\to \infty}\max_{1\leq i\leq N}
        \bE \Big[\,\sup_{0\leq t\leq T}\left|X_t^{i}-Z_t^i\right|^2 \Big] = 0.
\end{align*}
    
The established uniform in $N$ upper bound on $\bE[|D_s^j X_r^{i}|^2]$, given by  \eqref{eq:gronwallindividualparticle}, allows us to conclude that  $Z^i\in \bD^{1,2}(\cS^2)$ and $D^j_s Z^i_t = \lim_{N\to\infty}D^j_s X^{i}_t$ as prescribed by \cite{nualart2006malliavin}*{Lemma 1.2.3}. 
\medskip

\textit{Step 4. Recovering the limiting equation and conclusion.} 
It remains only to identify and confirm the stochastic differential equation the limiting object of $D^j_s X^{i}_t$ as $N\to \infty$ satisfies. 

We now consider the $L^2$-limit of the right hand side of \eqref{malderiv}.  First, note that the MV-SDE \eqref{Zmd} is an affine SDE with random coefficients that are space-measure Lipschitz, hence existence and uniqueness follows from \cite{mao2008stochastic}*{Theorem 2.1}.
Define $\delta\sigma^l_s := \sigma^l(s,X_s^{i},\bar{\mu}_s^N)-\sigma^l(s,Z_s^i,\mu_s^i)$, $l=1,\ldots,m$ and $\delta\sigma_s$ to be the $d \times m$ matrix valued process with columns $\delta\sigma^l_s$. By our Lipschitz assumption on $\sigma^l$, for all $s\in[0,T]$:
\begin{align}
\label{deltasigmaPOC}
    |\delta\sigma_s|\leq L\Big(|X_s^{i}-Z_s^i|+W_2(\bar{\mu}_s^N,\mu_s^i)\Big) \to 0,  
\end{align} as $N\to \infty$ in $L^2(\Omega)$ by Proposition \ref{lemma:PoCforBSigmaSystem}. 
Computing the difference between the SDEs \eqref{Zmd} and \eqref{malderiv},
\begin{align*}
    D_s^j & X_t^i  
    - D_s^jZ_t^i 
    = \delta\sigma_s
    +\int_s^t \theta_r \dd r 
    + \int_s^t \Big\{ \frac{1}{N}\sum_{k=1}^{N}(\partial_\mu b)(r,X_r^{i},\bar{\mu}_r^N)(X^{k}_r)D_s^j X_r^{k} \Big\} \dd r
    \\
    &
    +\sum_{l=1}^m\int_s^t \eta_r^l \dd W_r^{l,i} 
    + \sum_{l=1}^m \int_s^t  \Big\{ \frac{1}{N}\sum_{k=1}^{N}(\partial_\mu \sigma^l)(r,X_r^{i},\bar{\mu}_r^N)(X^{k}_r)D_s^j X_r^{k} \Big\} \dd W_r^{l,i}, 
\end{align*}
where
\begin{align*}
    \theta_r &:= (\nabla_x b)(r,X_r^i,\bar{\mu}^N_r)D_s^jX_r^i - (\nabla_x b )(r,Z_r^i,\mu_r^i)D_s^j Z_r^i , 
\\
    \eta_r^l &:= (\nabla_x \sigma^l)(r,X_r^i,\bar{\mu}^N_r)D_s^jX_r^i - (\nabla_x \sigma^l) (r,Z_r^i,\mu_r^i)D_s^j Z_r^i.
\end{align*}
Squaring both sides of the equation, expanding the squares and taking expectations, we have 
\begin{align}
\label{differencemalderiv}
    \lVert D_s^jX_t^i-D_s^jZ_t^i\rVert_{L^2(\Omega)}^2 -\Phi_N =  \Big\lVert\int_s^t \theta_r \dd r\Big\rVert_{L^2(\Omega)}^2 + \Big\lVert\sum_{l=1}^m\int_s^t \eta_r^l \dd W_r^{l,i} \Big \rVert_{L^2(\Omega)}^2, 
\end{align}
where we define
\begin{align*}
    \Phi_N 
    :=\bE\Big[|\delta\sigma_s&
    |^2\Big]+2\bE
    \Big[\delta\sigma_s : \int_s^t \theta_r \dd r\Big]
    \\
    &
    +2\bE\Big[\delta\sigma_s : \int_s^t\frac{1}{N}\sum_{k=1}^{N}(\partial_\mu b)(r,X_r^{i},\bar{\mu}_r^N)(X^{k}_r)D_s^j X_r^{k} \dd r\Big]
    \\
    &
+2\bE\Big[\int_s^t \theta_r \dd r :\int_s^t\frac{1}{N}\sum_{k=1}^{N}(\partial_\mu b)(r,X_r^{i},\bar{\mu}_r^N)(X^{k}_r)D_s^j X_r^{k} \dd r\Big] 
\\
&
+ 2\bE\Big[\sum_{l=1}^m\int_s^t \eta_r^l\dd r : \int_s^t\frac{1}{N}\sum_{k=1}^{N}(\partial_\mu \sigma^l)(r,X_r^{i},\bar{\mu}_r^N)(X^{k}_r)D_s^j X_r^{k} \dd r\Big]
\\
&
+\bE\Big[\Big|\int_s^t\frac{1}{N}\sum_{k=1}^{N}(\partial_\mu b)(r,X_r^{i},\bar{\mu}_r^N)(X^{k}_r)D_s^j X_r^{k} \dd r\Big|^2\Big]
\\
&+\bE\Big[\sum_{l=1}^m\int_s^t\Big|\frac{1}{N}\sum_{k=1}^{N}(\partial_\mu \sigma^l)(r,X_r^{i},\bar{\mu}_r^N)(X^{k}_r)D_s^j X_r^{k}\Big|^2\dd r\Big].
\end{align*}
The term $\Phi_N$ is a sequence which converges to zero as $N \to\infty$. This is justified by repeated use of the Cauchy-Schwarz and Jensen inequalities, over the bounds \eqref{eq:convergenceofmeanmallderiv} and \eqref{deltasigmaPOC}; the $L^2(\Omega)$ boundedness of $\theta_r$ and $\eta_r$ is justified by the Lipschitz property of $\sigma^l$ and $b$ (i.e. $ \nabla_x \sigma^l$ and $\nabla_x b$ are uniformly bounded maps). 
The remaining term on the left hand side of \eqref{differencemalderiv} converges to zero by the conclusion of \cite{nualart2006malliavin}*{Lemma 1.2.3} (shown in Step 3 of this proof). Hence so must the right hand side of \eqref{differencemalderiv} and we obtain that 
\begin{align*}
\lim_{N\to\infty}\int_{s}^{t} (\nabla_x b)(r,X_r^{i},\bar{\mu}_r^N)D_s^jX_r^{i}\, dr = \int_{s}^{t}(\nabla_x b)(r,Z^i_r,\mu_r^i)D_s^j Z^i_r \, dr, 
\end{align*} 
and 
\begin{align*}
\lim_{N\to\infty}\int_{s}^{t} (\nabla_x \sigma^l)(r,X_r^{i},\bar{\mu}_r^N)D_s^jX_r^{i}\, \dd W_r^i = \int_{s}^{t}(\nabla_x \sigma^l)(r,Z^i_r,\mu_r^i)D_s^j Z^i_r \dd W_r^{l,i} .
\end{align*} 
in the $L^2$-sense and the it follows that MV-SDE \eqref{Zmd} is identified as the limit of \eqref{malderiv}.

\end{proof}

\begin{remark}[Mollification: Lifting Assumption \ref{Assump:All is Lipschitz }(2.); the measure differentiability requirement]

We do not carry out the analysis here, but by drawing on techniques of mollification in Wasserstein spaces it seems possible to remove the measure differentiability assumption of Proposition \ref{malliavinIPS}. 
There are several ways to carry out mollification over the space of measures, with the main difficulty being that the Wasserstein space of measures is an infinite dimensional one. 

\textit{[Inf-Sup convolution]:} 
Motivated by the study of Hamilton-Jacobi equations in infinite dimensional spaces, Lasry-Lions \cite{LasryLions1986-infsupconvolution} propose inf-sup convolution (in Hilbert spaces) to show that bounded uniformly continuous scalar functions defined on a Hilbert space $\mathcal{H}$ can be uniformly approximated by functions belonging to the class of differentiable maps with Lipschitz derivatives. Critically, their methodology preserves, in the infinite dimensional space, certain good properties that other known mollifications at the time were not known to. As an example, in \cite{daudin2023optimal} and working on the torus, the  the inf-sup convolution techniques of \cite{LasryLions1986-infsupconvolution} are used to establish optimal rates for the convergence problem in mean field control. Notably, the mollification procedures of \cite{daudin2023optimal} rely on the properties of the Hilbert Sobolev space $H^{-s}$ for $s>d / 2+1$, i.e., the dual of the Hilbert space of functions with $s$ generalized derivatives in $L^2$.

\textit{[Smoothing via truncating Fourier expansions]:} 
In \cite{cecchin2022weak}, working over the probability measure $\cP(\bT^d)$ on the $d$-dimensional torus $\bT^d$, a mollification procedure of real-valued functions on $\cP(\bT^d)$ based on the Fourier coefficients of the measure is introduced. Very roughly, the mollification is carried out by truncating higher-order coefficients of the Fourier expansion at a conveniently chosen threshold and combining with a Fej\'er kernel (see their Definition 3.13). It is currently open how to extend the analysis from $\bT^d$ to $\bR^d$.

\textit{[Smooth approximation over $L^\infty$ under $W_1$]:}
In \cite{mou2019wellposedness}*{Section 3}, the authors construct a novel mollification technique for measure functionals in $C(\cP_1(\bR^d))$ under the $1$-Wasserstein distance and the approximation is carried in $L^\infty$ (the subset of  bounded RVs with norm $\lVert X \rVert_{L^\infty(\Omega)}=\esssup_{\omega\in\Omega}|X(\omega)|$). 
Critically, their mollified map satisfies uniformly the same Lipschitz property of the original functional. Such property does not hold under $W_2$ (unless the functional is also $W_1$-Lipschitz) but a convergence result is provided (see their Theorem 3.1).

\textit{[Smoothing the approximating particle projection]:}
The interacting particle system point of view allows us to benefit from a finite dimensional framework and draw from  standard mollification arguments (in $(\bR^d)^N$). 
The most suitable regularisation method for our IPS approach has been shown in \cite{chassagneux2014probabilistic} (also \cites{book:CarmonaDelarue2018a,book:CarmonaDelarue2018b} and the complementary \cite{platonovdosreis2023ItowentzelLions} adding all the missing details of the proofs of \cites{book:CarmonaDelarue2018a,book:CarmonaDelarue2018b}). 
There, regularization is applied directly to the empirical projection map $u^N$ (of Definition \ref{def:Auxiliary-uN-for-EmpirialTrick} and not $u$) via convolution with a smooth kernel. The method allows to control the derivatives of the mollified map by the Lipschitz constant. 

In \cite{cosso2023smooth}, the authors close the two open questions left by \cite{mou2019wellposedness}. They show that when $u$ is $W_2$-continuous there exists a sequence $\{u_k\}_k \in C^{\infty}(\mathscr{P}_2(\mathbb{R}^d))$ converging to $u$ uniformly on compact subsets of $\mathscr{P}_2(\mathbb{R}^d)$. If $u$ is additionally uniformly (resp.~Lipschitz) continuous then each $u_k$ is also uniformly (resp.~Lipschitz) continuous, with the same modulus of continuity (resp. Lipschitz constant) as $u$ -- this closes \cite{mou2019wellposedness}*{Remark 3.2(i)}. Moreover, for $u \in C^1(\mathscr{P}_2(\mathbb{R}^d))$ (resp.~$u \in C^2(\mathscr{P}_2(\mathbb{R}^d))$) we show that the convergence also holds for the first-order derivative (resp.~second-order derivatives) -- this closes \cite{mou2019wellposedness}*{Remark 3.2(ii)}.
The smooth approximating sequence $\{u_k\}_k \in C^{\infty}(\mathscr{P}_2(\mathbb{R}^d))$ of \cite{cosso2023smooth} is constructed relying on the empirical distribution, similarly to what is done in \cite{book:CarmonaDelarue2018a}*{Theorem 5.92} in the proof of It\^o's formula along a flow of measures, although there it is not really a smoothing as both $u$ and $u_k$ are of class $C^2(\mathscr{P}_2(\mathbb{R}^d)$), but rather a way to approximate a function $u$ on $\mathscr{P}_2(\mathbb{R}^d)$ by functions defined on finite-dimensional spaces. 

\end{remark}

\section*{Acknowledgement}

G. dos Reis acknowledges partial support from the FCT – Fundação para a Ciência e a Tecnologia, I.P., under the scope of the projects UIDB/00297/2020 (\url{https://doi.org/10.54499/UIDB/00297/2020}) and UIDP/00297/2020 (\url{https://doi.org/10.54499/UIDP/00297/2020}) (Center for Mathematics and Applications, NOVA Math), and by the UK Research and Innovation (UKRI) under the UK government’s Horizon Europe funding Guarantee [grant number UKRI343].









\end{document}